\documentclass[12pt,twoside,a4paper]{amsart}

\usepackage[T1]{fontenc}
\usepackage[latin1]{inputenc}
\usepackage[left=25mm,top=30mm,bottom=30mm]{geometry}
\usepackage{amssymb}
\usepackage{mathtools}
\usepackage{lmodern}
\usepackage[
final,
scrtime
]{prelim2e}

      \theoremstyle{plain}
      \newtheorem{thm}{Theorem}[section]
      \newtheorem{lem}[thm]{Lemma}
      \newtheorem{cor}[thm]{Corollary}
      
      \theoremstyle{definition}

      \theoremstyle{remark}
      
\let\oldsqrt\sqrt
\def\sqrt{\mathpalette\DHLhksqrt}
\def\DHLhksqrt#1#2{%
\setbox0=\hbox{$#1\oldsqrt{#2\,}$}\dimen0=\ht0
\advance\dimen0-0.2\ht0
\setbox2=\hbox{\vrule height\ht0 depth -\dimen0}%
{\box0\lower0.4pt\box2}}
   
\begin{document}


   \author[Dinev]{Todor Dinev}
   \address{Universit\"at Trier, Fachbereich IV -- Mathematik, 54296~Trier, Germany}
   \email{\{dinev, mattner\}@uni-trier.de}


    \author[Mattner]{Lutz Mattner}



   

   \title
{The asymptotic Berry-Esseen constant for intervals}


   \begin{abstract}
The asymptotic constant in the Berry-Esseen inequality for interval 
probabilities is shown to be $\sqrt{2/\pi}$.  
   \end{abstract}



\thanks{This research was partially supported by DFG grant MA 1386/3-1.}
\thanks{File name: \texttt{\tiny\jobname.tex}}



   \maketitle



   \section{Introduction}
Let 
$\mathbb R$ denote the 
real numbers, $\mathbb Z$ the integers, and
$\mathbb N$ the integers $\geq1$.  Let $\text{Prob}(\mathbb R)$ be the set of
all (probability) laws on $\mathbb R$, $\delta_x$ the unit mass at $x\in\mathbb
R$, and $\mathrm N_{0, 1}$ the standard normal 
law with distribution function
$\Phi$  and density $\varphi\coloneqq\Phi'$.

For $P\in\text{Prob}(\mathbb R)$ and 
$s\in\mathopen[1, \infty\mathclose[$, 
let $\mu(P)\coloneqq\int x\,P(\mathrm dx)$ and 
$\beta_s(P)\coloneqq\int|x-\mu(P)|^sP(\mathrm dx)$
if $\int |x|\,P(\mathrm dx)<\infty$,
$\beta_s(P)\coloneqq\infty$ otherwise,
$\sigma^2(P)\coloneqq\beta_2(P)$, 
and  $\alpha(P)\coloneqq\int(x-\mu(P))^3P(\mathrm dx)$ 
if $\beta_3(P)<\infty$.
 If $P$ is a lattice law, let $h(P)\coloneqq\sup\bigcup_{a\in\mathbb
   R}\{\eta\in\mathopen]0, \infty\mathclose[:P(a+\eta\mathbb Z)=1\}$,
 otherwise set $h(P)\coloneqq0$.  Let us put
\begin{gather*}
  \mathcal P_s\coloneqq \{P\in\text{Prob}(\mathbb R):0<\beta_s(P)<\infty\}.
\end{gather*}

In the following, we shall consider sequences $(X_k)_{k\in\mathbb N}$ of 
independent and identically distributed real-valued random variables.  
$P$ will then stand for the law of $X_1$ and shall, for the sake of brevity, 
usually be omitted in the quantities $\mu$, $\sigma^2$, $\alpha$, $\beta_s$, and $h$ just defined.  Let $P_n\in\mathrm{Prob}(\mathbb R)$ be given by $P_n(A)\coloneqq P^{*n}(\sigma\sqrt nA+n\mu)$ and set $F_n(x)\coloneqq P_n(\mathopen]-\infty, x\mathclose])$ for $x\in\mathbb R$. Thus $P_n$ is the law of the standardized sum of $X_1, \ldots, X_n$ and $F_n$ its distribution function.  Now the classical Berry-Esseen theorem states the finiteness of the 
 \textbf{Berry-Esseen constant}
\begin{gather*}
  c_{\text{BE}}\coloneqq
\sup_{\substack{P\in\mathcal P_3\\n\in\mathbb N}}\frac{\sigma^3(P)}{\beta_3(P)}\sqrt n\sup_{x\in\mathbb R}|F_n(x)-\Phi(x)|.
\end{gather*}
It is known 
that $0.4097<(\sqrt{10}+3)/(6\sqrt{2\pi})
\leq c_{\text{BE}}<0.4748$,
with the lower bound due to Esseen~\cite{ess1956}
as an obvious consequence of his determination of the 
\textbf{asymptotic Berry-Esseen constant}
\begin{gather*}
  c_{\infty, \text{BE}}\coloneqq \sup_{P\in\mathcal P_3}\frac{\sigma^3(P)}{\beta_3(P)}\lim_{n\to\infty}\sqrt n\sup_{x\in\mathbb R}|F_n(x)-\Phi(x)|=\frac{\sqrt{10}+3}{6\sqrt{2\pi}},
\end{gather*}
and with the more recent upper bound due to Shevtsova~\cite{she2011}. 

The Kolmogorov distance used above is the supremum distance over the system of all unbounded intervals.  We propose to use instead the system $\mathfrak I$ of all intervals whatsoever, which is a special case of the system of all convex sets used in the multidimensional setting.  Due to the continuity properties of probability measures, it suffices to consider intervals of the form $\mathopen]a, b\mathclose]$, that is, we have
\begin{gather}\label{eq:convex}
  \sup_{I\in\mathfrak I}|P_n(I)-\mathrm N_{0, 1}(I)|=\sup_{x, y\in\mathbb R}|F_n(x)-F_n(y)-\Phi(x)+\Phi(y)|.
\end{gather}
An obvious upper bound for the right hand side is
$2c_{\text{BE}}\beta_3/(\sigma^3\sqrt n)$.  While we do not know whether 
here $2c_{\text{BE}}$ can be replaced by a smaller constant, we will determine
below the 
\textbf{asymptotic Berry-Esseen constant for interval probabilities}
as 
\begin{gather}
\label{eq:CBE.convex}
  c_{\infty, \text{BE}}(\mathfrak I)\coloneqq 
\sup_{P\in\mathcal P_3}\frac{\sigma^3(P)}{\beta_3(P)}
 \lim_{n\to\infty}\sqrt n\sup_{I\in\mathfrak I}|P_n(I)-\mathrm N_{0, 1}(I)|
\,=\, \frac2{\sqrt{2\pi}},
\end{gather}
which is strictly smaller than $2c_{\infty, \text{BE}}$.
Let us note that  $c_{\infty, \text{BE}}(\mathfrak I)$ happens to be 
twice the modified asymptotic Berry-Esseen constant
due to  Rogozin~\cite{Rogozin}, 
and a consideration of the symmetric Bernoulli case shows that 
a Rogozin type modification of  $c_{\infty, \text{BE}}(\mathfrak I)$
yields the same value $\sqrt{2/\pi}$ as above.

\section{Results}
For  $P\in\mathcal P_3$,  we shall use the following asymptotic expansion 
due to  Esseen~\cite{ess1945}:
\begin{gather}\label{eq:asy}
  F_n(x)=\Phi(x)+\frac{h}{\sigma\sqrt n}\psi_n(x)\varphi(x)-\frac{\alpha}{6\sigma^3\sqrt{n}}\varphi''(x)+o\left(\frac1{\sqrt n}\right)\qquad(n\to\infty)
\end{gather}
uniformly in $x\in\mathbb R$, where, if $h>0$, $\psi_n$ is 
a certain  $h/(\sigma\sqrt n)$-periodic function: 
If $a\in\mathbb R$ is such that $P(a+h\mathbb Z)=1$, then
$\psi_n(x)=1/2-\mathrm{frac}(x\sigma\sqrt n/h-an/h)$, where
$\mathrm{frac}(t)\coloneqq t-\lfloor t\rfloor$ for $t\in\mathbb R$.  As 
observed in~\cite{ess1956}, it then  follows that 
\[\lim_{n\to\infty}\sqrt n\sup_{x\in\mathbb R}|F_n(x)-\Phi(x)|
=\frac1{\sqrt{2\pi}}\left(\frac h{2\sigma}+\frac{|\alpha|}{6\sigma^3}\right).\]
The analogous result for interval probabilities is:
\begin{thm}\label{thm:limit}
    Let $P\in\mathcal P_3$.  Then
  \begin{multline*}
    \lim_{n\to\infty}\sqrt n\sup_{I\in\mathfrak I}|P_n(I)-\mathrm N_{0,
      1}(I)|\\=\frac1{\sqrt{2\pi}} 
    \begin{cases} 
      \frac h\sigma&\text{if }|\alpha|\leq h\sigma^2,\\ 
\frac
h{2\sigma}+\frac{|\alpha|}{6\sigma^3}+\frac{|\alpha|}{3\sigma^3}\exp\Bigl(-\frac32\bigl(1-\frac{h\sigma^2}{|\alpha|}\bigr)\Bigr)&\text{if
}|\alpha|> h\sigma^2.
    \end{cases}
  \end{multline*}
\end{thm}

\begin{proof}
We may assume  $\alpha\geq0$, $\sigma=1$, and $\mu=0$.  
On the right of~\eqref{eq:convex}, 
we may  omit the absolute value signs, since 
for a function $f$ satisfying $f(x, y)=-f(y, x)$, 
we have $\sup_{x, y}|f(x, y)|=\sup_{x, y}f(x, y)\vee(-f(x, y))
=\sup_{x, y}f(x, y)\vee\sup_{x, y}f(y, x)=\sup_{x, y}f(x, y)$.
Thus, using~\eqref{eq:asy}, we get (reading backwards
to prove existence of the limits)
\begin{align*}
 \lim_{n\rightarrow\infty} 
2&\sqrt n\sup_{I\in\mathfrak I}|P_n(I)-\mathrm N_{0, 1}(I)| \\
 &\,=\, \lim_{n\rightarrow\infty} \sup_{x,y\in\mathbb R}\left(
  2h\,\bigl(\psi_n(x)\varphi(x)-\psi_n(y)\varphi(y)\bigr)
 -\frac{\alpha}3\bigl(\varphi''(x)-\varphi''(y)\bigr)\right)\\
 &\,=\,\sup_{x,y\in\mathbb R}\left( h\varphi(x)-\frac{\alpha}3\varphi''(x)
  +h\varphi(y)+\frac{\alpha}3\varphi''(y)\right) \\
 &\,=\, \sup_{y\in\mathbb R} f(y)
\end{align*}
where  $f(y) \coloneqq  \frac1{\sqrt{2\pi}}\left(h+\alpha/3\right)
  +h\varphi(y)+\alpha\varphi''(y)/3$ and where,
if $h>0$, 
the second equality above is true since 
$\sup_{x\in\mathbb R}\psi_n(x)=1/2$,  
$\inf_{x\in\mathbb R}\psi_n(x)=-1/2$, 
the period length of $\psi_n$ tends to zero,
and the function occurring on the right is continuous.

If $\alpha=0$, then $f$ is maximized at $y=0$ with the value
$f(0)=2h/\sqrt{2\pi}$.  
Let now $\alpha>0$. Then  
$f'(y)=h\varphi'(y)+\alpha\varphi'''(y)/3
=-{\alpha}y\varphi(y)(y^2-3+3h/\alpha)/3$.  
If also $\alpha\leq h$, then $\{f'>0\}=\mathopen]-\infty, 0\mathclose[$, and
$f$ is again maximized at $y=0$, with the same value as above.
If instead $\alpha>h$, then $\{f'>0\}=\mathopen]-\infty,
-y_0\mathclose[\cup\mathopen]0, y_0\mathclose[$ with $y_0\coloneqq
\sqrt{3-3h/\alpha}$, 
and   the maximal value of $f$ is 
  \begin{gather*}
f(\pm y_0)=\frac1{\sqrt{2\pi}}\Biggl(h+\frac{\alpha}3+\frac{2\alpha}3\exp\biggl(-\frac32\Bigl(1-\frac h{\alpha}\Bigr)\biggr)\Biggr).\qedhere
  \end{gather*}
\end{proof}

For completeness we present a short proof of a classical fact needed below:

\begin{lem}[von Mises \cite{mis1938}]\label{lem:von-mises}
Let $P\in\mathrm{Prob}(\mathbb R)$ be discrete and let $\eta \in\mathopen[0,
\infty\mathclose[$ be such that $0<|x-y|<\eta$ implies 
$P(\{x\}) P(\{y\}) =0$. Let  $s\in\mathopen[1, \infty\mathclose[$.
Then
  \begin{gather}\label{eq:von-mises}
    \eta\beta_s\leq2\beta_{s+1}.
  \end{gather}
Finite equality $\eta\beta_s=2\beta_{s+1}<\infty$ holds iff $s=1$ and 
$P=\lambda\delta_x+(1-\lambda)\delta_{x+\eta}$ for some $\lambda\in[0, 1]$ and
$x\in\mathbb R$, or $s>1$ and 
$P\in\bigcup_{x\in\mathbb R}\{\delta_x, (\delta_x+\delta_{x+\eta})/2\}$.
\end{lem}
\begin{proof}
We may assume 
$\eta>0$ and $P\in\mathcal P_{s+1}\setminus 
\{(\delta_x+\delta_{y})/2 :x, y\in\mathbb R\}$, 
all other cases being trivial, and then also $\mu=0$.
Let $X,Y$ be independent with law $P$.
Then $|X-\mu|=|X|$ is not constant almost surely, 
and hence the map $[1,s+1]\ni t \mapsto \log \beta_t \in\mathbb R$ is strictly
convex by H\"older's inequality, yielding  
$\beta_{s+1}/\beta_s\ge \beta_2/\beta_1$, with equality iff $s=1$.
We also have $2\beta_2 = 
\mathbb E(X-Y)^2\ge \eta\mathbb E|X-Y|\geq \eta\mathbb E|X|=\eta\beta_1$,
with equality in the first inequality iff  
$P$ is of the form  
$ \lambda\delta_x+(1-\lambda)\delta_{x+\eta}$,
in which case equality holds throughout.
Hence \eqref{eq:von-mises} holds, with the 
discussion of equality.
\end{proof}

\begin{cor} Identity~\eqref{eq:CBE.convex} is true, 
with the $\sup_{P\in\mathcal P_3}$ 
attained precisely for $P\in\{(\delta_{x}+\delta_{y})/2:x, y\in\mathbb R, x\neq y\}$.
\end{cor}
\begin{proof}
Let $P\in\mathcal P_3$ and let $L$ denote the limit in  
Theorem~\ref{thm:limit}.
If $|\alpha|\leq h\sigma^2$, then 
Lemma~\ref{lem:von-mises} with $s=2$
and $\eta=h$
yields 
$\sqrt{2\pi}L=h/\sigma\leq2\beta_3/\sigma^3$,
with equality in the last step iff $P$ is of  the form 
$(\delta_{x}+\delta_{y})/2$. 
If $|\alpha|>h\sigma^2$, 
then $\sqrt{2\pi} L 
<h/(2\sigma)+|\alpha|/(6\sigma^3)+|\alpha|/(3\sigma^3)
<|\alpha|/\sigma^3\leq\beta_3/\sigma^3$.
\end{proof}

\end{document}